\documentclass[conference]{IEEEtran}
\IEEEoverridecommandlockouts
\usepackage{cite}
\usepackage{amsmath,amssymb,amsfonts}
\usepackage{algorithmic}
\usepackage{graphicx}
\usepackage{textcomp}
\usepackage{xcolor}
\def\BibTeX{{\rm B\kern-.05em{\sc i\kern-.025em b}\kern-.08em
    T\kern-.1667em\lower.7ex\hbox{E}\kern-.125emX}}
\newtheorem{theo}{Theorem}
\newtheorem{lem}{Lemma}

\newtheorem{de}{Definition}

\begin{document}

\title{Decentralized design of consensus protocols with minimal communication links based on directed spanning tree\\
\thanks{This work is supported by Beijing Natural Science Foundation (4232041) and National Natural Science Foundation (NNSF) of China under Grant 62273014.}
}

\author{
	\IEEEauthorblockN{
		Yangzhou Chen, 
		Lanhao Zhao} 
	College of Artificial Intelligence and Automation, Beijing University of Technology\\
		Engineering Research Center of Digital Community, Ministry of Education
		Beijing, China\\
		E-mail: yzchen@bjut.edu.cn, zhaolanhao@emails.bjut.edu.cn.
} 

\maketitle

\begin{abstract}
This paper proposes a decentralized design approach of consensus protocols of multi-agent systems (MASs) via a directed-spanning-tree(DST)-based linear transformation and the corresponding minimal communication links. First, the consensus problem of multi-agent systems is transformed into the decentralized output stabilization problem by constructing a linear transformation based on a DST of the communication topology, and thus a necessary and sufficient consensus criterion in terms of decentralized fixed mode (DFM) is derived. Next, a new distributed protocol is designed by using only the neighbors’ information on the DST, which is a fully decentralized design approach. Finally, some numerical examples are given to verify the results attained.
\end{abstract}

\begin{IEEEkeywords}
Consensus, multi-agent system, decentralized output feedback stabilization, decentralized fixed mode, minimal communication links
\end{IEEEkeywords}

\section{Introduction}
With the development of information technology, especially network technology, the MAS has been paid more and more attention in recent years. Various cooperative control problems such as consensus \cite{b1}, controllability \cite{b2}, and stabilization \cite{b3} have been studied. Especially, the consensus design methods have been widely used such as \cite{b4,b5,b6} et al.

Despite the fruitful results that have been obtained, two key issues require further attention. One is to find decentralized design approaches for the distributed consensus protocols, and the other is to minimize the amount of communication between agents. For example, communicating with all neighbors at the same time generates a large amount of data and may cause congestion. Two ways are adopted to reduce the amount of communication data: one is to adopt discontinuous communication, and the other is to communicate with as few neighbors as possible. For discontinuous communication, the event-triggering mechanism has been proposed \cite{b8} and several results have been obtained such as\cite{b9}. In comparison, there is less work to reduce the amount of communication data by reducing the number of communication neighbors.

Obviously, the less neighbor information used in the protocol, the less the amount of communication data, and the higher the reliability of communication. As well-known, the existence of a DST in a fixed communication topology is a necessary condition for MAS achieving consensus \cite{b7}. This implies that the communication based on a DST possesses the minimum number of communication links. This fact motivates us to propose a protocol design approach using only DST-based communication.
On the other hand, in existing literature, despite the fact that the structure of consensus protocols is expressed in distributed form, the parameters in the protocols are generally determined by centralized design procedures, or they are limited to specific ranges for obtaining decentralized design procedures. For instance, the same control gain matrices and even the scalars are required for all the agents for ease of adjustment, which greatly reduces the freedom of parameter selection. Therefore, it is of great practical significance to seek novel approaches to consensus protocol design directly in a decentralized way. Different from the traditional approach of converting the consensus problem into a stability problem by taking the state error, the linear transformation method \cite{b10,b11,b12} is proposed to transform the consensus problem of multi-agent systems into the decentralized output stabilization problem. For example, a new consensus condition can be derived in terms of the concept of DFM in the decentralized output stabilization problem, and a fully decentralized design approach can be proposed for finding the different gain matrices. 

This paper studies the consensus problem of multi-agent systems based on the linear transformation method. First, by constructing a DST-based linear transformation, the consensus problem of multi-agent systems is transformed into the decentralized output stabilization problem. Next, using only the neighbor information of the agent on the DST, rather than all the neighboring agents, a new distributed protocol is designed and proved that the system can reach state consensus. Finally, some numerical examples are given to verify the results.

The rest of this paper is organized as follows. Section 2 presents the problem description and transformation based on a DST-based linear transformation. Section 3 provides a fully decentralized protocol design method that uses minimal communication links. Simulations illustrate the validity of the theoretical results in Section 4. Finally, Section 5 summarizes the results attained.
\section{Problem description and transformation}
In this section, we first give a description of the state consensus problem for the MAS  and use a DST-based linear transformation to equivalently transform the consensus problem into a corresponding decentralized output stabilization problem. Based on this, a necessary and sufficient condition for the MAS to reach state consensus is attained by using the concept of the DFM in decentralized output-feedback stabilization theory.

Consider a general linear MAS, where the dynamics of each agent is
\begin{equation}
	\dot{x}_i=A x_i+B u_i, i=1, \cdots, N, x_i \in \mathbb{R}^n, u_i \in \mathbb{R}^m
\end{equation}
Suppose that the communication relationship between agents is described by a weighted directed graph $\mathcal{G}=(V, E, W)$, where $V=\{1,2, \cdots, N\}, E \subseteq\{(j, i): i, j \in V\}$. $(j, i) \in E$ indicates that the agent $j$ can transfer information to the agent $i$, or that the agent $j$ is the neighbor of the agent $i$. We use $N_i$ to represent the collection of all neighbors of agent $i$. $W=\left[w_{i j}\right] \in \mathbb{R}_{+}^{N \times N}$ is the adjacency weight matrix, where $w_{i j} \geq 0$, and $w_{i j}>0$ if and only if $(j, i) \in E$.

Under the above communication topology, the agent $i$ can obtain the following information (either directly by measuring the relative state $x_j-x_i$ through the agent $i$, or by transferring information from agent $j$ to the agent $i$ through the communication network)
\begin{equation}
z_{j i}=w_{i j}\left(x_j-x_i\right), \quad j \in N_i
\end{equation}
or a combination of this information
\begin{equation}
z_i=\sum_{j \in N_i} z_{j i}=\sum_{j \in N_i} w_{i j}\left(x_j-x_i\right)
\end{equation}
Therefore, the problem of distributed state consensus protocol design can be described as follows: for each agent in (1) a linear feedback control protocol is designed based on its available information (2) or (3) as follows
\begin{equation}
	u_i=K_i z_i=K_i \sum_{j \in N_i} z_{j i}, i=1, \cdots, N
\end{equation}
so that the closed-loop system composed of (1) and (4) meets
\begin{equation}
	\begin{gathered}
		\lim _{t \rightarrow \infty}\left\|x_i(t)-x_j(t)\right\|=0 \\
		\forall i, j \in\{1, \cdots, N\}, \forall x_i(0), x_j(0) \in \mathbb{R}^n
	\end{gathered}
\end{equation}
If the gain matrices $K_i \in \mathbb{R}^{m \times n}$ exist, the MAS (1) is said to achieve asymptotically state consensus under the control protocol (4) with the communication topology $\mathcal{G}$.
For any agent $i \in V$ and its neighbor $j \in N_i$, we call $x_{j i}=x_j-x_i$ the edge state of the directed edge $(j, i) \in E$. Thus, the edge state system is represented as
\begin{equation}
	\dot{x}_{j i}=A x_{j i}+B u_j-B u_i, \quad i \in V, j \in N_i
\end{equation}
It is assumed that the communication topology $\mathcal{G}=(V, E, W)$ has a $\operatorname{DST} \mathcal{T}=\left(V, E_{\mathcal{T}}\right)$, and for the convenience of expression, it is assumed that the vertex $N$ is the root node, and each non-root node $i \in\{1,2, \cdots, N-1\}$ has only one parent node $k_i$, that is $e_i:=\left(k_i, i\right) \in E_{\mathcal{T}}$, $i=1, \cdots, N-1$. We call $e_i, i=1, \cdots, N-1$, the fundamental edges corresponding to the DST $\mathcal{T}$. The edge state $y_i:=x_{k_i}-x_i \quad i=1, \cdots, N-1$, corresponding to the fundamental edge $e_i$ are called the fundamental edge state based on the DST, and thus $y=\left[y_1^T, \cdots, y_{N-1}^T\right]^T$ is the fundamental edge state vector.
\begin{lem}
	Any edge state $x_{j i}, i \in V, j \in N_i$ can be linearly represented by the fundamental edge states $y_i=x_{k_i}-x_i,i=1, \cdots, N-1$ on the DST.
\end{lem}
\begin{IEEEproof}
	As we know, any two vertices $i, j$ have a common grandparent node (such as the root node $N$ ) in $\mathcal{T}$. Consider the grandparent node $k$ with two shortest paths, say, $\left(k, l_1\right),\left(l_1, l_2\right) \cdots,\left(l_\alpha, i\right)$ and $\left(k, m_1\right),\left(m_1, m_2\right) \cdots,\left(m_\beta, j\right)$ respectively, then
	\begin{equation}
			\begin{aligned}
			x_{j i} &=\left(x_k-x_i\right)-\left(x_k-x_j\right) \\
			&=x_{k l_1}+x_{l_1 l_2}+\cdots+x_{l_\alpha i}-x_{k m_1}-x_{m_1 m_2}-\cdots-x_{m_\beta j} \\
			&=:\left(\gamma_{j i} \otimes I_n\right) y
		\end{aligned}
	\end{equation}
	where $\gamma_{j i}$ is the $N-1$ dimensional row vector with components in $\{-1,0,1\}$. The Proof is completed.
\end{IEEEproof}

Therefore, instead of the complete edge dynamic system (6), it is only necessary to consider the following fundamental edge state system (assume $N$ to be the root node)
\begin{equation}
	\dot{y}_i=A y_i-B u_i+B u_{k_i}, \quad i=1, \cdots, N-1
\end{equation}
We construct matrix $\Gamma_i=\left[\begin{array}{lll}\gamma_{1 i}^T & \cdots & \gamma_{N i}^T\end{array}\right]^T \in \mathbb{R}^{N \times(N-1)}$, where the row vector $\gamma_{j i}$ is given by (7) when $j \in N_i$, and addedly define $\gamma_{j i}=0$ when $j \notin N_i$. It can be seen that the matrix $\Gamma_i$ represents the information flow relationship that can be obtained by the agent $i$ corresponding to the DST $\mathcal{T}$, which is called the information flow matrix of the agent $i$ with respect to the DST.

Let $P_0=\left[p_{i j}\right] \in \mathbb{R}^{N \times(N-1)}$ be the incidence matrix corresponding to the DST $\mathcal{T}$, and its entries are defined as $$p_{i j}= \begin{cases}1, & \text { Node } i \text { is the starting point of edge } j \text { (out-edge) } \\ -1, & \text { Node } i \text { is the end point of edge } j \text { (in-edge) } \\ 0, & \text { Node } i \text { is not associated with edge } j\end{cases}$$ If the agents are renumbered appropriately, it can be assumed that the parent node $k_i$ of each non-root node $i \in\{1,2, \cdots, N-1\}$ satisfies $k_i>i$, that is, $P_0$ is a lower triangular matrix with diagonal element $-1$, element 1 in $\left(k_i, i\right)$ and zero in all other entries.
Thus, transpose of the incidence matrix $P_0$
\begin{equation}
	P_0^T=:\left[\begin{array}{lll}
		p_1 & \cdots & p_N
	\end{array}\right] \in \mathbb{R}^{(N-1) \times N}
\end{equation}
is the upper triangular matrix, where the $N-1$ dimension column vector $p_i=\left[p_{i 1}, \cdots, p_{i, N-1}\right]^T$ represents the association relationship between the node $i$ and all fundamental edges: $p_{i i}=-1$ implies that the fundamental edge $e_i=\left(k_i, i\right)$ is the in-edge of the vertex $i$; for $j \in\{1,2, \cdots, N-1\} \backslash\{i\}$, the fundamental edge $e_j$ is an out-edge of vertex $i$ if $p_{i j}=1$; and the fundamental edge $e_j$ is not associated with the vertex $i$ if $p_{i j}=0$.

Now the fundamental edge state system (8) can be written as a centralized system
\begin{equation}
	\begin{aligned}
		\dot{y}=\left(I_{N-1} \otimes A\right) y+\sum_{i=1}^N\left(p_i \otimes B\right) u_i
	\end{aligned}
\end{equation}
Furthermore, letting $w_i=\left[\begin{array}{lll}
	w_{i 1} & \cdots & w_{i N}
\end{array}\right] \in \mathbb{R}^{1 \times N}$
be the $ith$  row vector of the adjacency matrix $W=\left[w_{i j}\right]$, the information (3) is expressed as
\begin{equation}
	z_i=\sum_{j \in N_i} z_{j i}=\sum_{j \in N_i} w_{i j}\left(x_j-x_i\right)=\left(w_i \Gamma_i \otimes I_n\right) y
\end{equation}
It can be seen that the state consensus problem of MAS
(1) under the control protocol (4) is equivalent to the problem of decentralized output-feedback stabilization problem: design decentralized output-feedback control based on measured output (11) for system (10)
\begin{equation}
	u_i=K_i z_i, i=1, \cdots, N
\end{equation}
to make the closed-loop system
\begin{equation}
	\dot{y}=\left[I_{N-1} \otimes A+\sum_{i=1}^N\left(p_i w_i \Gamma_i \otimes B K_i\right)\right] y
\end{equation}
is asymptotically stable.
For the convenience of expression, the system (13) can be simply expressed as $\left(C^*, A^*, B^*\right)$, i.e.,
\begin{equation}
	\left\{\begin{array}{l}
		\dot{y}=A^* y+B^* u \\
		z=C^* y
	\end{array}\right.
\end{equation}
where
$$
\begin{aligned}
	A^* &=I_{N-1} \otimes A, 
	B^* =\left[B_1, \cdots, B_N\right]=P_0^T \otimes B \\
	C^* &=\left[C_1^T, \cdots, C_N^T\right]^T ,B_i =p_i \otimes B,\\ 
	C_i&=w_i \Gamma_i \otimes I_n, i=1, \cdots, N
\end{aligned}
$$
The centralized form of control (12) is expressed as
\begin{equation}
	\begin{aligned}
		u=& K_{\mathrm{D}} z, K_{\mathrm{D}} \in \mathbf{K}_{\mathrm{D}} \\
		\mathbf{K}_{\mathrm{D}}:=&\left\{K_{\mathrm{D}}=\operatorname{diag}\left(K_1, \cdots, K_N\right):\right.\\
		&\left.K_i \in \mathbb{R}^{m \times n}, i=1, \cdots, N\right\}
	\end{aligned}
\end{equation}
So the system matrix of the closed-loop system (13) is
\begin{equation}
	\begin{aligned}
		&I_{N-1} \otimes A+\sum_{i=1}^N\left(p_i w_i \Gamma_i \otimes B K_i\right) \\
		&=I_{N-1} \otimes A+\sum_{i=1}^N B_i K_i C_i \\
		&=A^*+B^* K_{\mathrm{D}} C^*
	\end{aligned}
\end{equation}
It is well known that the results of decentralized output stabilization are closely related to DFMs. Next, we introduce the concept of DFM to analyze the stability of the system (13).
\begin{de}
	The set of decentralized fixed modes for system $\left(C^*, A^*, B^*\right)$ with the block diagonal feedback gain matrix set $\mathbf{K}_{\mathrm{D}}$ in (15) is defined as
	$\Lambda\left(C^*, A^*, B^* ; \mathbf{K}_{\mathrm{D}}\right)=\bigcup_{K \in \mathbf{K}_{\mathrm{D}}} \sigma\left(A^*+B^* K C^*\right)$
	where $\sigma\left(A^*+B^* K C^*\right)$ denotes the set of eigenvalues of matrix $A^*+B^* K C^*$.
\end{de}

The following lemma can be established from Definition 1.
\begin{lem}
	MAS (1) under protocol (4) can asymptotically achieve state consensus if and only if the DFMs of system (14) does not contain the one in
	$\left\{\lambda: \lambda \in \sigma\left(A^*\right), \operatorname{Re} \lambda \geq 0\right\}$.
\end{lem}

According to \cite{b13} and Lemma 2, a sufficient and necessary condition for MAS to achieve state consensus has been obtained in \cite{b12}.
\begin{theo}
	MAS(1) under protocol (4) can asymptotically achieve state consensus if and only if
	$$
	\operatorname{rank}\left[\begin{array}{cc}
		I_{N-1} \otimes\left(\lambda_0 I-A\right) & P_\alpha \otimes B \\
		\Phi_\beta \otimes I_n & 0
	\end{array}\right] \geq(N-1) n
	$$
	for any $\lambda_0 \in\{\lambda: \lambda \in \sigma(A), \operatorname{Re} \lambda \geq 0\}$ and any bipartite segmentation $(\alpha, \beta)$ of the vertex set $V$, where
	$$
	P_\alpha=\left[p_{i_1}, \ldots, p_{i_k}\right], \Phi_\beta=\left[\begin{array}{c}
		w_{i_{k+1}} \Gamma_{i_{k+1}} \\
		\vdots \\
		w_{i_N} \Gamma_{i_N}
	\end{array}\right]
	$$
\end{theo}
Next, we will discuss a fully decentralized design of the distributed consensus protocol by using only some but not all of the neighbor information.
\section{Decentralized consensus protocol design via minimal communication links}
In this section, we propose a fully decentralized design approach for the consensus protocol of MAS (1). Precisely, instead of the control protocol (12) where the information of all neighbors of each agent is used, we show that, due to the unique role the DST plays in the consensus process, consensus can also be reached by using only the information that an agent obtains from their neighbors on the DST. For the convenience of discussion, we distinguish between two cases: 1) the root node of the DST is also a root vertex in the original communication topology; 2) the root node of the DST is not a root vertex in the original communication topology.

Under the assumption that there is a DST $\mathcal{T}=\left(V, E_\mathcal{T}, W_\mathcal{T}\right)$, if the control protocol of each agent only uses the neighbor information on the DST, the information flow matrix $\Gamma_N=0$, and for $i \in\{1,2, \cdots, N-1\}$ only the $ith$ element of the vector in $\Gamma_i$ is 1, and the other vectors are 0, i.e. 
$w_i \Gamma_i=\left[\begin{array}{lll}0  \cdots 0& w_{i, k_i} &0 \cdots0\end{array}\right]$. Thus, the control (12) becomes 
\begin{equation}
	\begin{aligned}
		&u_N=0 \\
		&u_i=K_i w_{i, k_i} y_i, i=1, \cdots, N-1
	\end{aligned}
\end{equation}
\begin{theo}
	Suppose the root node $N$ of the DST $\mathcal{T}$ is a root vertex in the communication topology $\mathcal{G}$. If the control gain matrix $K_i, i=1,2, \cdots, N-1$ in (17) are selected such that $A-w_{i k_i} B K_i, i=1,2, \cdots, N-1$ are Hurwitz, then MAS (1) asymptotically achieves state consensus under the control protocol (17).
\end{theo}
\begin{IEEEproof}
	Set $p_{i j}$ as the $j$ th row element of the vector $p_i$, $\Gamma_{i k}$ is the $k$ th column vector of the matrix $\Gamma_i$, thus we have $ p_i w_i \Gamma_i=\left[p_{i j} w_i \Gamma_{i k}\right]_{j, k=1}^{N-1}$. Noting that $p_{i j}=-1, p_{k_j j}=1, p_{i j}=0$ for the rest $i$; and $w_i \Gamma_{i k_i}=w_{i, k_i}, w_i \Gamma_{i k}=0$ for the rest $k$, then (13) can be expressed in the form
	\begin{equation}
		\begin{aligned}
			\dot{y}_j &=A y_j+\sum_{i=1}^N \sum_{k=1}^{N-1}\left(p_{i j} w_i \Gamma_{i k}\right) B K_i y_k \\
			&=\left(A-w_{j k_j} B K_j\right) y_j+w_{k_j \kappa\left(k_j\right)} B K_{k_j}  y_{k_j}
		\end{aligned}
	\end{equation}
	where $\kappa\left(k_j\right)$ represents the parent node of the node $k_j$ on the DST and $j=1, \cdots, N-1$.
	Since it is assumed that $k_j>j$ for $j \in\{1,2, \cdots, N-1\}$, the system matrix of (13) is an upper triangular block matrix according to (18), and $A-w_{j k_j} B K_j, j=1,2, \cdots, N-1$ is the block matrix on the diagonal. So MAS (1) asymptotically achieve state consensus under the control protocol (17) if the control gain matrix $K_i, i=1,2, \cdots, N-1 \quad$ are selected such that $A-w_{i k_i} B K_i, i=1,2, \cdots, N-1$ is Hurwitz. The Proof is completed.
\end{IEEEproof}

Next, we deal with the general case where the root node $N$ of the DST is not a root vertex in the original communication topology. Let $n_l \in\{1,2, \cdots, N-1\}$ be one of the neighbors of root node $N$ in the original communication topology, and $\left(N, n_1\right),\left(n_1, n_2\right) \cdots,\left(n_{l-1}, n_l\right)$ be the shortest path from $N$ to $n_l$ on the DST. Then any edge state can be linearly represented by the fundamental edge states  according to Lemma 1, so we have
\begin{equation}
	w_{N, n_l}\left(x_N-x_{n_l}\right)=w_{N, n_l}\left(y_{n_1}+y_{n_2}+\cdots+y_{n_l}\right)
\end{equation}
Thus we can design the following control protocol
\begin{equation}
	\begin{aligned}
		&u_N=K_N w_{N_i n_l}\left(y_{n_1}+y_{n_2}+\cdots+y_{n_l}\right) \\
		&u_i=K_i w_{i, k_i} y_i, \quad i=1, \cdots, N-1
	\end{aligned}
\end{equation}
Suppose $\alpha$ is a row block matrix with blocks $-w_{N, n_l} B K_N$ at the columns $n_1, \cdots, n_l$ and zero at the other columns and $\beta_i$ is row block matrix with blocks $w_{k_k,\left(k_1\right)} B K_{k_1}$ at the columns $k_i$ and zero at the other columns. Suppose there are $p-1$ nodes whose parent node on DST is $N$. Without loss of generality, let $k_i \neq N$ for $i=1, \cdots, N-p$ and $k_i=N$ for $i=N-p+1, \cdots, N-1$. Thus, in this case, we have
\begin{equation}
	A^*+B^* K_{\mathrm{D}} C^*=\left[\begin{array}{c}
		\Phi+\beta^* \\
		\alpha^*+Q
	\end{array}\right]
\end{equation}
where
$$\Phi:=\operatorname{diag}\left\{A-w_{i i_i} B K_i, i=1, \cdots, N-p\right\}$$
$$\Psi=\operatorname{diag}\left\{A-w_{i N} B K_i, i=N-p+1, \cdots, N-1\right\}$$
$$Q=\left[\begin{array}{ll}0_{n p \times n(N-p-1)} & \Psi\end{array}\right]$$
$$\left.\beta^*:=\left[\begin{array}{c}
	\beta \\
	\vdots \\
	\beta_{N-1-p}
\end{array}\right], \alpha^*:=\left[\begin{array}{c}
	\alpha \\
	\vdots \\
	\alpha
\end{array}\right]\right\} \text { there are } p \text { blocks }$$ 
Obviously, when all eigenvalues of the matrix in (21) are located in the left half plane, MAS (1) asymptotically achieves state consensus under the control protocol (19). However, the size of the eigenvalues of matrix (21) belongs to the global information, and we expect to use only local information to obtain criteria. Next, we use Gerschgorlin circle theorem\cite{b14} to deal with this problem.

For any matrix $A=\left[a_{i j}\right] \in \mathbb{C}^{m \times n}$, we denote that 
$$
\|A\|=\max _{1 \leq i \leq m} \sum_{j=1}^n\left|a_{i j}\right|,\left(\left\|A^{-1}\right\|\right)^{-1}=\min _{1 \leq i \leq m} \sum_{j=1}^n\left|a_{i j}\right|
$$
Then we can get the following results.
\begin{theo}
	Suppose the root node $N$ of the DST is not a root vertex in the original communication topology and let $n_l \in\{1,2, \cdots, N-1\}$ be one of its neighbors and $\lambda$ are the eigenvalues to be estimated of matrix (21).
	If for $p-1$ nodes whose parent node on DST is $N$, the gain matrices $K_i,i=N-p+1,..., N $ in (20) are selected such that the following Gerschgorlin circles
	\begin{equation}
		\left(\left\|\left(A-w_{i N} B K_i-\lambda I_n\right)^{-1}\right\|\right)^{-1} \leqq \sum_{m=1}^l\left\|-w_{N, n_m} B K_N\right\|
	\end{equation}
	are all located in the left half plane, and for the remaining $N-p$ nodes, the gain matrices $K_i, i=1,2,..., N-p$ in (20) are selected such that the following Gerschgorlin circles
	\begin{equation}
		\left(\left\|\left(A-w_{i k_i} B K_i-\lambda I_n\right)^{-1}\right\|\right)^{-1} \leqq\left\|w_{k_i \kappa\left(k_i\right)} B K_{k_i}\right\|
	\end{equation}
	are all located in the left half plane, then MAS (1) asymptotically achieves state consensus under the control protocol (19).
\end{theo}
\begin{IEEEproof}
We note that the matrix (21) is a block matrix, so the Gerschgorlin circles can be denoted as (22) and (23) according to the Gerschgorlin circle theorem of block matrix \cite{b14}, and we note that for each individual, we can design its feedback gain matrix and the neighbor's (on the DST) feedback gain matrix instead of designing all the gain matrices at the same time to adjust the position of the corresponding Gerschgorlin circle so that it is located in the left half plane. And as long as all Gerschgorlin circles are located in the left half plane, all eigenvalues of the matrix are also located in the left half plane. So we can show that the matrix (21) is Hurwitz if the control gain matrix $K_i, i=1,2, \cdots, N$ satisfies the conditions in Theorem 3, and thus MAS (1) asymptotically achieves state consensus under the control protocol (19). The proof is completed.
\end{IEEEproof}

We will further show the selection method in Example 2 in Section 4.

Theorems 2 and 3 show that MAS(1) can achieve state consensus even when designing a distributed protocol using only some neighbor information instead of all neighbor information. This greatly reduces the computation amount of the agents and the required communication load. It is worth noting that not only the protocol is distributed, but also its design procedure is fully decentralized and each agent may have a different feedback gain matrix. However, in most of the literature, all the agents adopt the same feedback gain matrix.

\section{Numerical examples}
This section presents some numerical examples to verify the validity of the proposed theoretical results.

Example 1. Consider a MAS consisting of four agents. The dynamics of each agent is described by
$$
\dot{x}_i=\left[\begin{array}{cc}
	0 & 1 \\
	-1 & 0
\end{array}\right] x_i+\left[\begin{array}{l}
	1 \\
	1
\end{array}\right] u_i
$$
The communication topology of the four agents is shown in Figure 1, and its adjacency matrix is
$$
W=\left[\begin{array}{llll}
	0 & 1 & 1 & 0 \\
	1 & 0 & 1 & 0 \\
	0 & 0 & 0 & 1 \\
	0 & 0 & 0 & 0
\end{array}\right]
$$
Obviously, it has a DST with fundamental edges $(4,3)(3,2)(3,1)$ and vertex 4 is a root node, $w_{3 k_3}=w_{34}=1, w_{2 k_2}=w_{23}=1, w_{1 k_1}=w_{13}=1$
\begin{figure}
	\centerline{\includegraphics[width=0.28\linewidth]{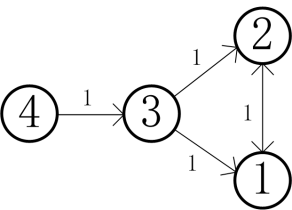}}
	\caption{The communication topology of Example 1}
	\label{fig}
\end{figure}
We select
$$
K_1=\left[\begin{array}{ll}
	1.5 & 0.5
\end{array}\right], K_2=\left[\begin{array}{ll}
	4 & 0
\end{array}\right], K_3=\left[\begin{array}{ll}
	2.625 & 0.375
\end{array}\right]
$$
Such that
$$
A-B K_1, A-B K_2, A-B K_3
$$
are Hurwitz. Figure 2 and Figure 3 show the states of the consensus procedure with initial states
$$
\begin{aligned}
	&x_{10}=7.5, y_{10}=13.8, x_{20}=14, y_{20}=9, \\
	&x_{30}=0, y_{30}=6.5, x_{40}=8, y_{40}=5.4 .
\end{aligned}
$$
\begin{figure}
	\centerline{\includegraphics[width=0.55\linewidth]{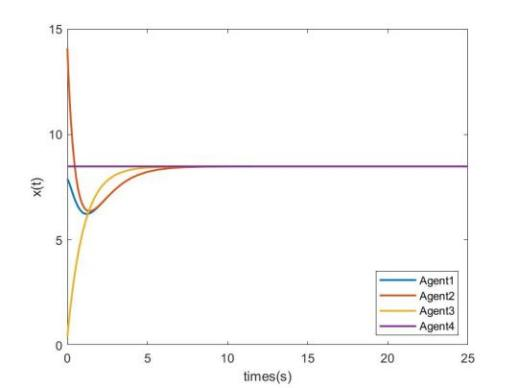}}
	\caption{The first components of the states}
	\label{fig}
\end{figure}
\begin{figure}
	\centerline{\includegraphics[width=0.55\linewidth]{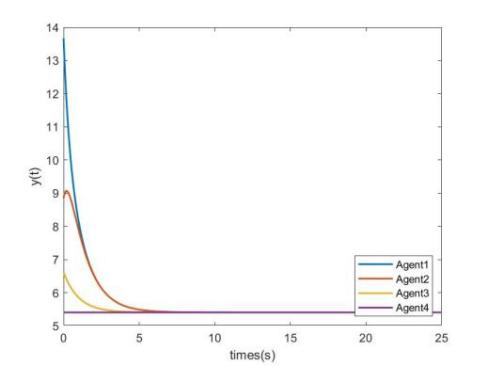}}
	\caption{The second components of the states}
	\label{fig}
\end{figure}

A more general example as follows which the root node of DST can receive neighbor information

Example 2. Consider a MAS consisting of six agents, and the dynamics of each agent is
$$
\dot{x}_i=\left[\begin{array}{cc}
	0 & 1 \\
	-1 & 0
\end{array}\right] x_i+\left[\begin{array}{c}
	1 \\
	1
\end{array}\right] u_i
$$
The communication topology of the agents is shown in Figure 4
\begin{figure}
	\centerline{\includegraphics[width=0.28\linewidth]{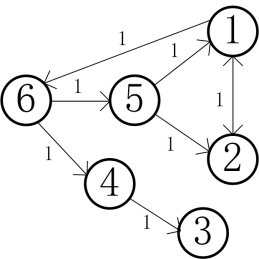}}
	\caption{ The communication topology of Example 2.}
	\label{fig}
\end{figure}
Obviously, it has a DST with fundamental edges $(6,5)(6,4)(4,3)(5,2)(5,1)$, where vertex 6 is the root node, and
$w_{5 k_5}=w_{56}=1, w_{4 k_4}=w_{46}=1, w_{3 k_3}=w_{34}=1$, $w_{2 k_2}=w_{25}=1, w_{1 k_1}=w_{15}=1$. Furthermore, vertex 6 is not the root node in the original communication topology and it has a neighbor 1 with $w_{61}=1$. In this case, the matrix in (21) is expressed as
\begin{small} 
	$$\left[\begin{array}{ccccc}A-B K_1 & 0 & 0 & 0 & B K_5 \\ 0 & A-B K_2 & 0 & 0 & B K_5 \\ 0 & 0 & A-B K_3 & B K_4 & 0 \\ -B K_6 & 0 & 0 & A-B K_4 & -B K_6 \\ -B K_6 & 0 & 0 & 0 & A-B K_5-B K_6\end{array}\right]$$
\end{small}
According to the conditions in Theorem 3, the following Gerschgorin circles all should be located in the left half-plane
$$
\begin{gathered}
	\left(\left\|\left(A-B K_i-\lambda I_n\right)^{-1}\right\|\right)^{-1} \leqq\left\|B K_{k_1}\right\|, i=1,2,3 \\
	\left(\left\|\left(A-B K_4-\lambda I_n\right)^{-1}\right\|\right)^{-1} \leqq 2 *\left\|-B K_6\right\| \\
	\left(\left\|\left(A-B K_5-B K_6-\lambda I_n\right)^{-1}\right\|\right)^{-1} \leqq\left\|-B K_6\right\|
\end{gathered}
$$
Next, we show the steps to design the feedback gain matrix separately for each individual compared with the case where the feedback gain matrices are the same. At first, we select the feedback gain matrix of node $\mathrm{N}$
$$
K_6=\left[\begin{array}{ll}
	1 & 0
\end{array}\right]
$$
Then we can select $K_4, K_5$ to make the Gerschgorin circle lie in the left half plane. 
$$
\left(\left\|\left(A-B K_5-B K_6-\lambda I_n\right)^{-1}\right\|\right)^{-1} \leqq 1
$$
$$
\left(\left\|\left(A-B K_4-\lambda I_n\right)^{-1}\right\|\right)^{-1} \leqq 2
$$
Next, the same method can be used to obtain $K_1, K_2, K_3$
In this example, we select
$K_1=[4\quad0]$, $K_2=[2.625\quad0.375]$, $K_3=[2.5\quad0.5]$, $K_4=[2.5\quad0.5]$, $K_5=[1.5\quad 0.5]$, $K_6=[1\quad0]$.

Figure 5 and Figure 6 show the states of the consensus procedure with initial states
$$
\begin{aligned}
	&x_{10}=7.8, y_{10}=4, x_{20}=10, y_{20}=7, x_{30}=19, y_{30}=18, \\
	&x_{40}=1, y_{40}=15, x_{50}=5.5, y_{50}=8.2 x_{60}=11, y_{60}=19
\end{aligned}
$$
\begin{figure}
	\centerline{\includegraphics[width=0.55\linewidth]{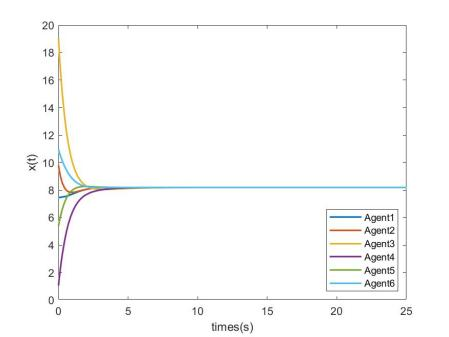}}
	\caption{The first components of the states}
	\label{fig}
\end{figure}
\begin{figure}
	\centerline{\includegraphics[width=0.55\linewidth]{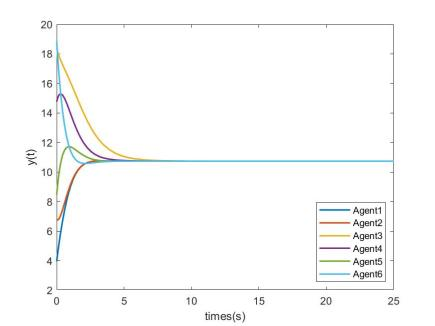}}
	\caption{The second components of the states}
	\label{fig}
\end{figure}
If we select $K_6=0$ and let $K_i, i=1,2,3,4,5$ keep unchanged, we can get the consensus procedure is shown in Figure 7 and Figure 8 with initial states
$$
\begin{aligned}
	&x_{10}=13, y_{10}=13.5, x_{20}=12.8, y_{20}=19, x_{30}=4.2, y_{30}=14,\\
	& x_{40}=4.8, y_{40}=2, x_{50}=12, y_{50}=9, x_{60}=9.1, y_{60}=13 .
\end{aligned}
$$
\begin{figure}
	\centerline{\includegraphics[width=0.55\linewidth]{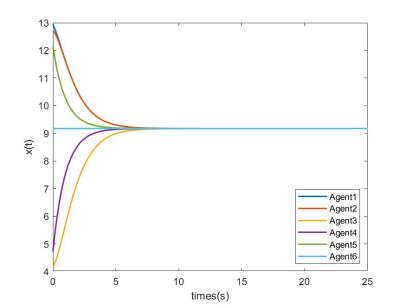}}
	\caption{The first components of the states}
	\label{fig}
\end{figure}
\begin{figure}
	\centerline{\includegraphics[width=0.55\linewidth]{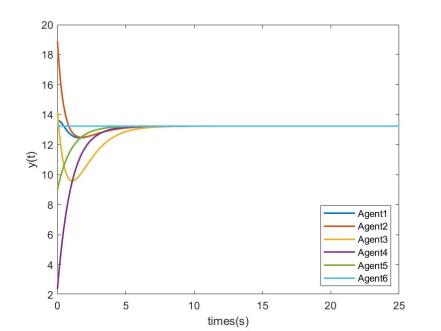}}
	\caption{The second components of the states}
	\label{fig}
\end{figure}
It can be seen from the simulations above that other agents achieve the tracking of the root node actually when the root node does not receive neighbor information.

\section{Conclusion}
The consensus problem of MASs based on the DST-based linear transformation method is studied. First, by constructing a DST-based linear transformation, the consensus problem of the multi-agent system is transformed into the decentralized output stabilization problem, and a sufficient and necessary condition for the multi-agent achieving state consensus was obtained by using the concept of DFM. Based on this, a fully decentralized design procedure of the distributed protocols was proposed by using only the neighbor information of each agent on the DST. We point out that the proposed method is generic and it can be applied to more complex consensus problems in future research.

\end{document}